\documentclass{article}
\usepackage{array}
\usepackage[caption=false,font=normalsize,labelfont=sf,textfont=sf]{subfig}
\usepackage{textcomp}
\usepackage{stfloats}
\usepackage{url}
\usepackage{verbatim}
\usepackage{cite}
\usepackage{amsthm,amsmath,amssymb,amscd,enumerate,epsfig}
\usepackage{amsfonts}
\usepackage{pst-all,epsfig,fancybox}
\usepackage{graphicx}
\usepackage{fancyhdr}
\usepackage{algorithmicx}
\usepackage{algorithm}
\usepackage[noend]{algpseudocode}
\usepackage{hyperref}
\usepackage{tikz}
\usetikzlibrary{calc}

\newtheorem{teo}{Theorem}

\newtheorem{lem}[teo]{Lemma}
\newtheorem{obs}{Remark}

\newcommand\DoubleLine[7][4pt]{
\path(#2)--(#3)coordinate[at start](h1)coordinate[at end](h2);
\draw[#4]($(h1)!#1!90:(h2)$)-- node [auto=left] {#5} ($(h2)!#1!-90:(h1)$);
\draw[#6]($(h1)!#1!-90:(h2)$)-- node [auto=right] {#7} ($(h2)!#1!90:(h1)$);}

\begin{document}
\textheight=570pt 
\begin{center} {\Large\bf A procedure to obtain symmetric cycles of any odd length using directed Haj\'os constructions}
\vskip 1cm {\bf Juan Carlos Garc\'ia-Altamirano, Mika Olsen, Jorge Cervantes-Ojeda }\\
\
\\Departmento de Matem\'aticas Aplicadas y Sistemas\\
Universidad Aut\'onoma Metropolitana - Cuajimalpa\\
M\'exico City, M\'exico\\
\
\\email: carlos\_treze@ciencias.unam.mx, olsen@cua.uam.mx, jcervantes@cua.uam.mx\\
\end{center}
\begin{center}
(Received January 17th, 2023)
\end{center}

%\maketitle
\begin{abstract}
The dichromatic number of a digraph $D$ is the minimum number of colors of a vertex coloring of $D$ such that $D$ has no monochromatic cycles.
The Haj\'os join were recently extended to digraphs (using the dichromatic number) by J. Bang-Jensen et. al. and Haj\'os (directed) operations is a tool to obtain r-(di)chromatic (di)graphs. J. Bang-Jensen et. al. posed in 2020 the problem of how to obtain  the symmetric cycle of length 5 from symmetric cycles of length 3. We recently solved this problem by applying a genetic algorithm. In this article, a procedure is presented to construct any odd symmetric cycle by applying  directed Haj\'os operations to symmetric cycles of length 3, thus, generalizing the known construction of the symmetric cycle of length 5. In addition, this procedure is analyzed to determine its computational complexity.
\end{abstract}

\section{Introduction}
%A vertex coloring of digraph is acyclic if there are no monochromatic directed cycles. 
The dichromatic number of a digraph $D$ was introduced by V. Neumann-Lara in 1982 \cite{Neumann1982} as an extension of the chromatic number of a graph. The dichromatic number of a digraph is the minimum number of colors of a vertex coloring of $D$ such that $D$ has no monochromatic cycles and several concepts and results for the chromatic number of a graph have been extended to digraphs using the dichromatic number. For instance, \cite{Andres,diachromatico,Narda,GLOR,Harut,Hoch}.
%
%A digraph $D$ is {\bf r-critical} if $dc(D)=r$ and $dc(H)<r$ for every proper subdigraph $H$ of $D$. 
In 2020 J. Bang-Jensen et. al. \cite{BangJensen2020} extended the well-known Haj\'os  join for graphs to digraphs.

The digraph $H$ obtained by {\bf identifying} a non-empty set $I$ of independent vertices is defined as the digraph $H=D-I$ adding a new vertex $v$  and adding all arcs from $v$ to $N^+_D(I)=\bigcup\limits_{u\in I} N^+_D(u)$ and all arcs from $N^-_D(I)=\bigcup\limits_{u\in I} N^-_D(u)$ to $v$. The new vertex $v$ may preserve the label of one of the vertices of the independent set $I$.

The Haj\'os join was defined for digraphs in 2020 by Bang-Jensen et. al. \cite{BangJensen2020} as an extension of the well-known Haj\'os  join \cite{Hajos,JR,Urquhart} for graphs. 
We use Figure \ref{Fig_Hajos} to illustrate the definition of directed Haj\'os Join. Let $D_1$ and $D_2$ be two disjoint digraphs. Let $u_1v_1\in A(D_1)$ and $v_2u_2\in A(D_2)$. The {\bf directed Haj\'os join} $D=(D_1,u_1,v_1)\triangledown(D_2,v_2,u_2)$ or, briefly $D=D_1\triangledown D_2$ of $D_1$ and $D_2$ is defined as the disjoint union of $D_1$ and $D_2$ and deleting both arcs $u_1v_1$ and $v_2u_2$, identifying the vertices $v_1$ and $v_2$ to a new vertex $v$ and adding the arc $u_1u_2$. The vertex $v$ may be denoted by $v_1$, $v_2$ or $v$.

\vspace{-0.7 cm}
\hspace{-0.5 cm}\begin{figure}[ht]
\beginpgfgraphicnamed{example}
\begin{center}
\begin{tikzpicture}[myn/.style={circle,very thick,draw,inner sep=0.11cm,outer sep=2.5pt},label distance=-1mm]

\node [myn,circle,label=45:$v_1$,fill=black,minimum size=0.55cm] 
(0) at (45.000000:1.1000) {};

\node [myn,circle,minimum size=0.55cm] 
(1) at (165.000000:1.1000) {};

\node [myn,circle,label=285:$u_1$,minimum size=0.55cm] 
(2) at (285.000000:1.1000) {};

\draw[very thick,black][->] (0) -- (1);
\draw[very thick,black][->] (1) -- (2);
\draw[very thick,black][dashed,->] (2) -- (0) node [pos=0.5,right,font=\footnotesize] {delete};

\pgftransformshift{\pgfpoint{4.00000cm}0cm}

\node [myn,circle,label=135:$v_2$,fill=black,minimum size=0.55cm] 
(3) at (135.000000:1.1000) {};

\node [myn,circle,label=255:$u_2$,minimum size=0.55cm] 
(4) at (255.000000:1.1000) {};

\node [myn,circle,minimum size=0.55cm] 
(5) at (15.000000:1.1000) {};

\draw[very thick,black][dashed,->] (3) -- (4) node [pos=0.5,left,font=\footnotesize] {delete};
\draw[very thick,black][->] (4) -- (5);
\draw[very thick,black][->] (5) -- (3);

\draw[line width = 2.0pt,black][->] (2) -- (4) node [pos=0.5,below,font=\footnotesize] {add};
\draw[line width = 0.0pt,white][-] (0) -- (3) node [pos=0.5,font=\footnotesize,black] {identify};

\pgftransformshift{\pgfpoint{5.00000cm}0cm}

\node [myn,circle,label=90:$v$,fill=black,minimum size=0.55cm] 
(0) at (90.000000:1.1000) {};

\node [myn,circle,minimum size=0.55cm] 
(1) at (162.000000:1.1000) {};

\node [myn,circle,label=234:$u_1$,minimum size=0.55cm] 
(2) at (234.000000:1.1000) {};

\node [myn,circle,label=306:$u_2$,minimum size=0.55cm] 
(3) at (306.000000:1.1000) {};

\node [myn,circle,minimum size=0.55cm] 
(4) at (378.000000:1.1000) {};

\draw[very thick,black][->] (0) -- (1);
\draw[very thick,black][->] (1) -- (2);
\draw[very thick,black][->, line width = 2.0pt] (2) -- (3);
\draw[very thick,black][->] (3) -- (4);
\draw[very thick,black][->] (4) -- (0);
\end{tikzpicture}
\end{center}
\endpgfgraphicnamed
\caption{Directed Haj\'os  join of two directed triangles and the resulting digraph.}
\label{Fig_Hajos}
\end{figure}
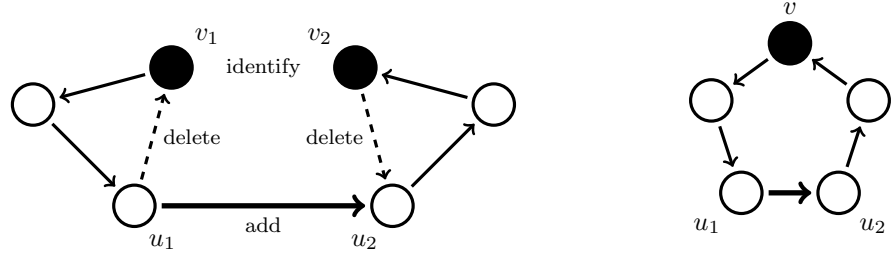

We call the directed Haj\'os join and the vertex identifications the {\bf directed Haj\'os operations}.  The class of {\bf Haj\'os-k-constructible digraphs} defined as the smallest family of digraphs that contains all complete digraphs of order $k$ and is closed under directed Haj\'os operations.  
J. Bang-Jensen et. al. \cite{BangJensen2020} proved that any $k$-critical digraph is Haj\'os-$k$-constructible. Since odd symmetric cycles are 3-critical, any symmetric odd cycle can be constructed from the symmetric complete digraph on three vertices $D(K_3)$ using a sequence of directed Haj\'os operations (directed Haj\'os joins and identifications of non-adjacent vertices). 
In the same paper they posed the following question (Question 19): 

\emph{How can a bidirected (symmetric) $C_5$ ($D(C_5)$) be constructed from copies of $D(K_3)$ by only using directed Haj\'os operations?}

In a recent paper J. C. Garc\'ia-Altamirano et. al. \cite{GarciaAlta} obtained a sequence of Haj\'os operations applied to $D(K_3)$ in order to obtain the symmetric cycle $D(C_5)$ using a Rank Genetic Algorithm. In this paper we generalize this sequence to obtain any odd symmetric cycle applying Haj\'os operations from  $D(K_3)$. 
This procedure provides an upper bound to the Haj\'os number of any odd symmetric cycle and an upper bound of its computational complexity. 

We consider finite digraphs without loops and multiple arcs. For all definitions not given here we refer the reader to the book of J. Bang-Jensen and G. Gutin  \cite{Conceptos}.  
Let $D$ be a digraph with vertex set $V (D)$ and arc set $A(D)$. The {\bf in-neighborhood} of a vertex $u$ is $N^-(u) = \{v\in V(D)\mid vu\in A(D)\}$  and the {\bf out-neighborhood} of a vertex $u$ is $N^+(u) = \{v\in V(D)\mid uv\in A(D)\}$. 
Two vertices in a digraph $D$ are {\bf independent} if there are no arcs between them in $D$, a set of vertices $X$ is independent in a digraph $D$ if any pair of vertices of $X$ are independent in $D$.
An arc $uv\in A(D)$ is {\bf symmetric} ({\bf asymmetric}) if $vu\in A(D)$ ($vu\notin A(D)$), and a digraph is symmetric (bidirected graph) if every arc of $D$ is a symmetric arc.
The {\bf symmetric digraph} $D(G)$, of the graph $G$, is the digraph obtained by replacing each edge by a symmetric arc.

\section{Construction of odd symmetric cycles}

Let $D$ and $D'$ be two disjoint digraphs of the same order $n$, with vertex-set $\{v_0,v_1,...,v_{n-1}\}$ and $\{v'_0,v'_1,...,v'_{n-1}\}$ resp., and let  $v_iv_j \in A(D)$ and $v'_kv'_l\in A(D')$, such that $i,j,k,l \in \{0,1,...,n-1\}$, $i\neq j$, $k\neq l$ and $j-i \neq k-l$. We define the \textbf{cyclic of Haj\'os identification} $H = (D,v_i,v_j)\otimes (D',v'_k,v'_l)$ or, simply $H = D\otimes D'$, as the digraph obtained by the directed Haj\'os Join $\hat{H} := (D, v_i, v_j) \triangledown (D',v'_k,v'_l)$, and the following $n-1$ identifications: for $r = 1,2,...,n-1$, the vertices $v'_{k+r}$ and $v_{j+r}$ are identified in the vertex $v_{j+r}$, the indices are taken {modulo} $n$.

\begin{obs}\label{Obs}
By definition, we identify the vertices $v'_a,v_{a-k+j} \in V(\hat{H})$ in the vertex $v_{a-k+j}$.
The condition  $j-i \neq k-l$ assures that the vertices $v'_l$ and $v_i$ are not identified, because in this case a loop is obtained in the vertex $v_i$ by the arc that is added  in $\hat{H}$. 
Note that two vertices in $\hat{H}$ which where original consecutive vertices in $D'$ are identified into two vertices which where originally two consecutive vertices in $D$. 
Moreover, if $v'_bv'_a \in A(\hat{H})$, in $H$ this arc becomes the arc $v_{b-k+j}v_{a-k+j}$ in $H$.
\end{obs}

\begin{lem}\label{Lem}
Let $n\ge1$ and let $D$ be a symmetric cycle  of order $2^{n+1} +1$ with asymmetric arcs $v_0v_{2^n}, v_{x}v_a$, where $x = 2^n+a+1$. Let  $D'$ be a copy of $D$ such that $v_i'\in V(D')$ is the copy of the vertex $v_i\in V(D)$. Then $H = (D,v_{x},v_a)  \otimes (D',v'_0,v'_{2^n})$ is a symmetric cycle of order $2^{n+1}+1$ with two asymmetric arcs $v_0v_{2^n}, v_{2^n+2a+1}v_{2a}$.
\end{lem}

\begin{proof}
Let  $H = (D,v_{x},v_a)  \otimes (D',v'_0,v'_{2^n})$, by Remark \ref{Obs},
\begin{description}
  \item $(i)$ the vertex $v'_{2^n}\in\hat{H}$ is identified with the vertex $v_{(2^n)-(0)+(a)} = v_{2^{n}+a} = v_{x-1}$, so the arc $v_{x}v'_{2^n}\in A(\hat{H})$ is transformed into the arc $v_{x}v_{x-1}\in A(H)$, that is, an arc between two vertices consecutive, therefore this arc is a symmetric arc.
  \item $(ii)$ The arc $v'_{x}v'_a$ de $\hat{H}$ is transformed into the arc $v_{(2^n+a+1)-0+a}v_{a-0+a} = v_{2^n + 2a + 1}v_{2a} = v_{x+a}v_{2a}$ of $H$.
\end{description}      
%Note that, the symmetric difference between $D$ and $H$ are the arcs $v_{x}v_a\in A(D)$ and the arc $v_{x+a}v_{2a}\in A(H)$, see Figura \ref{Fig_xal}.
Note that the digraphs $D$ and $H$ only differ in two arcs, namely $v_{x}v_a= A(D)\setminus A(H)$ and the arc $v_{x+a}v_{2a}= A(H)\setminus A(D)$, thus $H=D+\{v_{x+a}v_{2a}\} -\{v_{x}v_a\}$. In Figure \ref{Fig_xal} the dotted arrow represents the deleted arc $v_{x}v_a$.

\hspace{-0.5 cm}
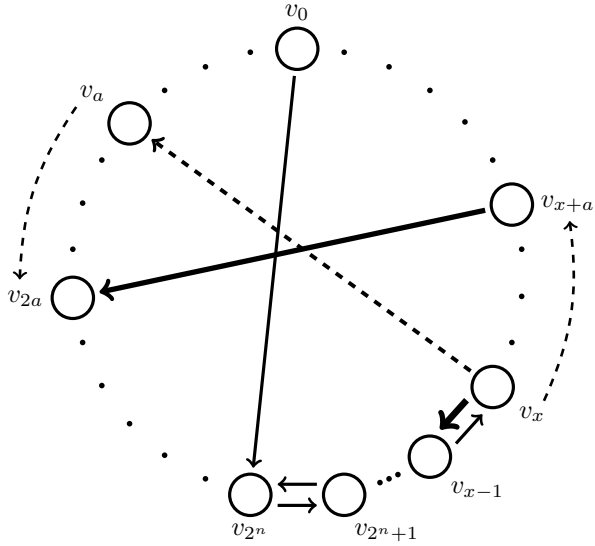
\begin{figure}[H]
\beginpgfgraphicnamed{H15}
\begin{center}
\begin{tikzpicture}[myn/.style={circle,very thick,draw,inner sep=0.11cm,outer sep=2.5pt},label distance=-1mm]

\node [myn,circle,label=90:$v_0$,minimum size=0.55cm] 
(0) at (90.000000:3.0000) {};

\node(1) at (114.000:3.0000) {\scalebox{2}{$\color{black} \cdot$}};
\node(15) at (102.000:3.0000) {\scalebox{2}{$\color{black} \cdot$}};
\node(16) at (126.000:3.0000) {\scalebox{2}{$\color{black} \cdot$}};

\node [myn,circle,label=138.000000:$v_{a}$,minimum size=0.55cm] 
(2) at (138.000000:3.0000) {};

\node(1) at (162.000:3.0000) {\scalebox{2}{$\color{black} \cdot$}};
\node(17) at (150.000:3.0000) {\scalebox{2}{$\color{black} \cdot$}};
\node(18) at (174.000:3.0000) {\scalebox{2}{$\color{black} \cdot$}};

\node [myn,circle,label=182.000000:$v_{2a}$,minimum size=0.55cm] 
(4) at (186.000000:3.0000) {};

\node (34) at (140.000000:3.7000)  {\scalebox{1}{}};
\node (35) at (183.000000:3.7000) {\scalebox{1}{}} edge [<-,dashed, line width = 1 pt,bend left=16] (34);

\node(5) at (210.000:3.0000) {\scalebox{2}{$\color{black} \cdot$}};
\node(6) at (234.000:3.0000) {\scalebox{2}{$\color{black} \cdot$}};
\node(19) at (222.000:3.0000) {\scalebox{2}{$\color{black} \cdot$}};
\node(30) at (198.000:3.0000) {\scalebox{2}{$\color{black} \cdot$}};
\node(31) at (246.000:3.0000) {\scalebox{2}{$\color{black} \cdot$}};

\node [myn,circle,label=270.000000:$v_{2^n}$,minimum size=0.55cm] 
(7) at (258.000000:3.0000) {};

\node [myn,circle,label=273.000000:$v_{2^n+1}$,minimum size=0.55cm] 
(8) at (282.000000:3.0000) {};

\node(9) at (48+306.000:3.0000) {\scalebox{2}{$\color{black} \cdot$}};
\node(20) at (48+294.000:3.0000) {\scalebox{2}{$\color{black} \cdot$}};
\node(21) at (48+318.000:3.0000) {\scalebox{2}{$\color{black} \cdot$}};

\node [myn,circle,label=330.000000:$v_{x}$,minimum size=0.55cm] 
(10) at (330.000000:3.0000) {};

\node(25) at (-48+342.000000:3.0000) {\scalebox{2}{$\color{black} \cdot$}};
\node(26) at (-48+339.5000000:3.0000) {\scalebox{2}{$\color{black} \cdot$}};
\node(27) at (-48+344.5000000:3.0000) {\scalebox{2}{$\color{black} \cdot$}};

\node [myn,circle,label=-48+354.000000:$v_{x-1}$,minimum size=0.55cm] 
(11) at (-48+354.000000:3.0000) {};

\node [myn,circle,label=362.000000:$v_{x+a}$,minimum size=0.55cm] 
(12) at (378.000000:3.0000) {};

\node (36) at (330.000000:3.7000)  {\scalebox{1}{}};
\node (37) at (373.000000:3.7000) {\scalebox{1}{}} edge [<-,dashed, line width = 1 pt,bend left=16] (36);

\node(13) at (402.000:3.0000) {\scalebox{2}{$\color{black} \cdot$}};
\node(14) at (426.000:3.0000) {\scalebox{2}{$\color{black} \cdot$}};
\node(29) at (414.000:3.0000) {\scalebox{2}{$\color{black} \cdot$}};
\node(32) at (390.000:3.0000) {\scalebox{2}{$\color{black} \cdot$}};
\node(33) at (438.000:3.0000) {\scalebox{2}{$\color{black} \cdot$}};

\draw[very thick,black][->] (0) -- (7);
\draw[very thick,black][dashed,->,line width = 1.5 pt] (10) -- (2);
\draw[very thick,black][->,line width = 2 pt] (12) -- (4);
\DoubleLine{7}{8}{<-,very thick}{}{->,very thick}{}
\DoubleLine{11}{10}{<-,very thick,line width = 2.5 pt}{}{->,very thick}{}

\end{tikzpicture}
\end{center}
\endpgfgraphicnamed
\caption{$H= D+\{v_{x+a}v_{2a}\} -\{v_{x}v_a\}$.}
\label{Fig_xal}
\end{figure}

 For $r = 2,\dots,n+1$, the digraph $H_r = (H_{r-1},v_{2^n+{2^{r-2}} + 1},v_{2^{r-2}})  \otimes (H'_{r-1},v'_0,v'_{2^n})$ is a symmetric cycle of order $2^{n+1}$ with the  asymmetric arcs $v_0v_{2^n}$, $v_{2^{n} + {2^{r-1}} + 1}v_{2^{r-1}}$.
\end{proof}

\begin{teo}\label{TheoNextCycle}
Let  $n\ge1$. The symmetric cycle of order $2^{n+1}+1$ can be constructed from the symmetric cycle  de order $2^{n}+1$ using $(n+2)\left(2^{n+1}+1\right)$ directed Haj\'os operations.
\end{teo}

\begin{proof}
Let  $n\geq 1$, consider the symmetric cycle of order $2^n+1$, $D = D(C_{2^n+1}) = (v_0,v_1,...,v_{2^n},v_0)$ and let $D'  = (v'_0,v'_1,...,v'_{2^n},v'_0)$ be a disjoint copy of $D$, where $v'_i$ is the copy of $v_i$ for $i \in \{0,1,...,2^n\}$. We define the digraph $H_0 = (D,v_{2^n},v_0)  \triangledown (D',v'_0,v'_1)$, see Figure \ref{Fig_H0}. %enemos que $v'_0$ se identifica con $v_0$. 

\hspace{-0.5 cm}
\begin{figure}[ht]
\beginpgfgraphicnamed{H0}
\begin{center}
\begin{tikzpicture}[myn/.style={circle,very thick,draw,inner sep=0.11cm,outer sep=2.5pt},label distance=-1.5mm]]

\node [myn,circle,label=90-72:$v_0$,fill=black,minimum size=0.55cm] 
(0) at (90.000000-72:1.1000) {};

\node [myn,circle,label=162.000000-72:$v_1$,minimum size=0.55cm] 
(1) at (162.000000-72:1.1000) {};

\node (2) at (234.000000-72:1.1500000) {\scalebox{2}{$\color{black} \cdot$}};
\node (10) at (216.000000-72:1.1500000) {\scalebox{2}{$\color{black} \cdot$}};
\node (11) at (252.000000-72:1.1500000) {\scalebox{2}{$\color{black} \cdot$}};

\node [myn,circle,label=306.000000-38:$v_{2^n-1}$,minimum size=0.55cm] 
(3) at (306.000000-72:1.1000) {};

\node [myn,circle,label=378.000000-100:$v_{2^n}$,minimum size=0.55cm] 
(4) at (378.000000-72:1.1000) {};

\DoubleLine{0}{1}{<-,very thick}{}{->,very thick}{}
\DoubleLine{0}{4}{dashed,<-,very thick}{}{->,very thick}{}
\DoubleLine{3}{4}{<-,very thick}{}{->,very thick}{}

\pgftransformshift{\pgfpoint{4.0000cm}0cm}

\node [myn,circle,label=90.000000+78:$v'_0$,fill=black,minimum size=0.55cm] 
(5) at (90.000000+72:1.1000) {};

\node [myn,circle,label=162.000000+72:$v'_1$,minimum size=0.55cm] 
(6) at (162.000000+72:1.1000) {};

\node [myn,circle,label=234.000000+50:$v'_2$,minimum size=0.55cm] 
(7) at (234.000000+72:1.1000) {};

\node (8) at (306.000000+72:1.1500000) {\scalebox{2}{$\color{black} \cdot$}};
\node (12) at (288.000000+72:1.1500000) {\scalebox{2}{$\color{black} \cdot$}};
\node (13) at (324.000000+72:1.1500000) {\scalebox{2}{$\color{black} \cdot$}};

\node [myn,circle,label=378.000000+72:$v'_{2^n}$,minimum size=0.55cm] 
(9) at (378.000000+72:1.1000) {};

\DoubleLine{5}{6}{<-,very thick}{}{dashed,->,very thick}{}
\DoubleLine{5}{9}{<-,very thick}{}{->,very thick}{}
\DoubleLine{6}{7}{<-,very thick}{}{->,very thick}{}

\draw[line width = 2.0pt,black][->] (4) -- (6);

\pgftransformshift{\pgfpoint{5.0000cm}0cm}

\node [myn,circle,label=90.000000:$v_{0}$,fill=black,minimum size=0.55cm] 
(0) at (90.000000:1.8000) {};

\node [myn,circle,label=130.000000:$v_{1}$,minimum size=0.55cm] 
(1) at (130.000000:1.8000) {};

\node (2) at (170.000000:1.8800) {\scalebox{2}{$\color{black} \cdot$}};
\node (9) at (160.000000:1.8800) {\scalebox{2}{$\color{black} \cdot$}};
\node (10) at (180.000000:1.8800) {\scalebox{2}{$\color{black} \cdot$}};

\node [myn,circle,label=260.000000:$v_{2^n-1}$,minimum size=0.55cm] 
(3) at (210.000000:1.8000) {};

\node [myn,circle,label=260.000000:$v_{2^n}$,minimum size=0.55cm] 
(4) at (250.000000:1.8000) {};

\node [myn,circle,label=290.000000:$v'_{1}$,minimum size=0.55cm] 
(5) at (290.000000:1.8000) {};

\node [myn,circle,label=350.000000:$v'_{2}$,minimum size=0.55cm] 
(6) at (330.000000:1.8000) {};

\node (11) at (380.000000:1.8800) {\scalebox{2}{$\color{black} \cdot$}};
\node (7) at (370.000000:1.8800) {\scalebox{2}{$\color{black} \cdot$}};
\node (12) at (360.000000:1.8800) {\scalebox{2}{$\color{black} \cdot$}};

\node [myn,circle,label=390.000000:$v'_{2^n}$,minimum size=0.55cm] 
(8) at (410.000000:1.8000) {};

\DoubleLine{4}{0}{<-,very thick}{}{dashed,->,very thick}{}
\DoubleLine{0}{5}{<-,very thick}{}{dashed,->,very thick}{}
\DoubleLine{0}{1}{<-,very thick}{}{->,very thick}{}
\DoubleLine{0}{8}{<-,very thick}{}{->,very thick}{}
\DoubleLine{3}{4}{<-,very thick}{}{->,very thick}{}
\draw[line width = 2.0pt,black][->] (4) -- (5);
\DoubleLine{5}{6}{<-,very thick}{}{->,very thick}{}
\end{tikzpicture}
\end{center}
\endpgfgraphicnamed
\caption{The directed Haj\'os joint of two symmetric cycles and the resulting digraph $H_0 = (D,v_{2^n},v_0)  \triangledown (D',v'_0,v'_1)$.}
\label{Fig_H0}
\end{figure}
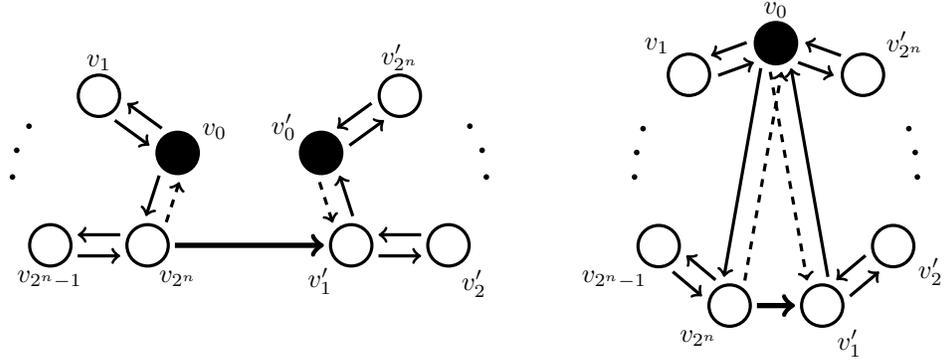

Relabel the vertices of $H_0$ as follows:

\begin{center}
$w = \left\{\begin{array}{ccc}
v_i & \text{if } & w = v_i \text{ for } i = 0,1,...,2^n,\\\\
v_{2^n + i} & \text{if } & w = v'_i \text{ for } i = 1,2,...,2^n.
\end{array}\right.$
\end{center}

Note that $H_0$ is a digraph of order $2^{n+1} + 1$, with asymmetric arcs  $v_{2^n}v_{2^n+1}$, $v_0v_{2^n}$, $v_{2^n+1}v_0$, and the symmetric path $v_{2^n+1},v_{2^n+2},\dots,v_{2^{n+1}},v_0,v_1,\dots,v_{2^n}$.
Let  $H'_0$ be a disjoint copy of $H_0$, with vertex-set $\{v'_0,v'_1,...,v'_{2^{n+1}}\}$ where $v'_i$ is the copy of the vertex $v_i$ for $i= 0,1,...,2^{n+1}$.

Let $H_1 = (H_0,v_0,v_{2^n})  \otimes (H'_0,v'_{2^n+1},v'_0)$.
In  Figure \ref{Fig_H1}, we consider the directed Haj\'os union and $2^{n+1}$ vertex identifications. The digraph on the left indicates the directed union of Haj\'os, where the dashed arrows must be removed, the thick arrow added, the black vertices are identified, and the shades of gray indicate the four pairwise vertex identifications. In each figure, the digraph on the right is the result of the $2^{n+1}+1$ Haj\'os operations.

By Remark \ref{Obs}:
\begin{description}
	\item $(i)$ The vertex $v'_0\in V(\hat{H}_1)$ is identified with the vertex $v_{(0)-(2^n+1)+(2^n)} = v_{2^{n+1}}$, hence, the arc $v_0v'_0\in A(\hat{H}_1)$, becomes the arc $v_0v_{2^{n+1}}\in A(H_1)$.
	\item $(ii)$ The arc $v'_0v'_{2^n}\in A(\hat{H}_1)$ becomes the arc $v_{(0)-(2^n+1)+(2^n)}v_{(2^n)-(2^n+1)+(2^n)} = v_{2^{n+1}}v_{2^n-1}$ in $H_1$.
\end{description}	

Thus, $H_1$ is a symmetric cycle of order $2^{n+1} + 1$ with asymmetric arcs $v_{2^{n+1}}v_{2^n-1}, v_{2^{n}+1}v_{0}$.

\hspace{-0.5 cm}
\begin{figure}[H]
\beginpgfgraphicnamed{H1p}
\begin{center}
\begin{tikzpicture}[myn/.style={circle,very thick,draw,inner sep=0.11cm,outer sep=2.5pt},label distance=-1.5mm]]

\node [myn,circle,label=90.000000:$v_{0}$,fill=black!15,minimum size=0.52cm] 
(0) at (90.000000:1.50000) {};

\node (26) at (84.000000:1.70000)  {\scalebox{1}{}}; 

\node [myn,circle,label=130.000000:$v_{1}$,fill=black!8,minimum size=0.52cm] 
(1) at (130.000000:1.50000) {};

\node (2) at (170.000000:1.5500) {\scalebox{2}{$\color{black} \cdot$}};
\node (9) at (160.000000:1.5500) {\scalebox{2}{$\color{black} \cdot$}};
\node (10) at (180.000000:1.5500) {\scalebox{2}{$\color{black} \cdot$}};

\node [myn,circle,label=260.000000:$v_{2^{n}-1}$,minimum size=0.52cm] 
(3) at (210.000000:1.50000) {};

\node [myn,circle,label=260.000000:$v_{2^n}$,fill=black,minimum size=0.52cm] 
(4) at (250.000000:1.50000) {};

\node [myn,circle,label=275.00000:$v_{2^n+1}$,fill=black!70,minimum size=0.52cm] 
(5) at (290.000000:1.50000) {};

\node [myn,circle,label=330.000000:$v_{2^n+2}$,fill=black!50,minimum size=0.52cm] 
(6) at (330.000000:1.50000) {};

\node (7) at (370.000000:1.550000) {\scalebox{2}{$\color{black} \cdot$}};
\node (11) at (360.000000:1.550000) {\scalebox{2}{$\color{black} \cdot$}};
\node (12) at (380.000000:1.550000) {\scalebox{2}{$\color{black} \cdot$}};

\node [myn,circle,label=410.000000:$v_{2^{n+1}}$,fill=black!25,minimum size=0.52cm] 
(8) at (410.000000:1.50000) {};

\draw[very thick,black][dashed,<-] (4) -- (0);
\draw[very thick,black][<-] (0) -- (5);
\DoubleLine{0}{1}{<-,very thick}{}{->,very thick}{}
\DoubleLine{0}{8}{<-,very thick}{}{->,very thick}{}
\DoubleLine{3}{4}{<-,very thick}{}{->,very thick}{}
\draw[very thick,black][->] (4) -- (5);
\DoubleLine{5}{6}{<-,very thick}{}{->,very thick}{}

\pgftransformshift{\pgfpoint{4.200cm}0cm}

\node [myn,circle,label=90.000000:$v'_0$,fill=black!25,minimum size=0.52cm] 
(13) at (90.000000:1.50000) {};

\node [myn,circle,label=130.000000:$v'_1$,fill=black!15,minimum size=0.52cm] 
(14) at (130.000000:1.50000) {};

\node [myn,circle,label=180.000000:$v'_2$,fill=black!8,minimum size=0.52cm] 
(15) at (170.000000:1.50000) {};

\node (16) at (210.000000:1.55000) {\scalebox{2}{$\color{black} \cdot$}};
\node (22) at (200.000000:1.55000) {\scalebox{2}{$\color{black} \cdot$}};
\node (23) at (220.000000:1.55000) {\scalebox{2}{$\color{black} \cdot$}};

\node [myn,circle,label=260.000000:$v'_{2^n}$,minimum size=0.52cm] 
(17) at (250.000000:1.50000) {};

\node [myn,circle,label=275.000000:$v'_{2^n+1}$,fill=black,minimum size=0.52cm] 
(18) at (290.000000:1.50000) {};

\node [myn,circle,label=330.000000:$v'_{2^n+2}$,fill=black!70,minimum size=0.52cm] 
(19) at (330.000000:1.50000) {};

\node [myn,circle,label=365.000000:$v'_{{2^n}+3}$,fill=black!50,minimum size=0.52cm] 
(20) at (370.000000:1.50000) {};

\node (21) at (410.000000:1.550000) {\scalebox{2}{$\color{black} \cdot$}};
\node (24) at (400.000000:1.550000) {\scalebox{2}{$\color{black} \cdot$}};
\node (25) at (420.000000:1.550000) {\scalebox{2}{$\color{black} \cdot$}};

\node (28) at (96.000000:1.70000)  {\scalebox{1}{}}; %(13)
\node (27) at (26) {\scalebox{1}{}} edge [->, line width = 2 pt,bend left=15] (28);

\draw[very thick,black][<-] (17) -- (13);
\draw[very thick,black][,dashed,<-] (13) -- (18);
\DoubleLine{13}{14}{<-,very thick}{}{->,very thick}{}
\DoubleLine{19}{20}{<-,very thick}{}{->,very thick}{}
\DoubleLine{14}{15}{<-,very thick}{}{->,very thick}{}
\draw[very thick,black][->] (17) -- (18);
\DoubleLine{18}{19}{<-,very thick}{}{->,very thick}{}

\pgftransformshift{\pgfpoint{5.60cm}0cm}

\node [myn,circle,label=90.000000:$v_0$,fill=black!15,minimum size=0.52cm] 
(0) at (90.000000:1.50000) {};

\node [myn,circle,label=130.000000:$v_1$,fill=black!10,minimum size=0.52cm] 
(1) at (130.000000:1.50000) {};

\node (2) at (170.000000:1.5500) {\scalebox{2}{$\color{black} \cdot$}};
\node (9) at (160.000000:1.5500) {\scalebox{2}{$\color{black} \cdot$}};
\node (10) at (180.000000:1.5500) {\scalebox{2}{$\color{black} \cdot$}};

\node [myn,circle,label=210.000000:$v_{2^{n}-1}$,minimum size=0.52cm] 
(3) at (210.000000:1.50000) {};

\node [myn,circle,label=270.000000:$v_{2^n}$,fill=black,minimum size=0.52cm] 
(4) at (250.000000:1.50000) {};

\node [myn,circle,label=275.000000:$v_{2^n+1}$,fill=black!70,minimum size=0.52cm] 
(5) at (290.000000:1.50000) {};

\node [myn,circle,label=330.000000:$v_{2^n+2}$,fill=black!50,minimum size=0.52cm] 
(6) at (330.000000:1.50000) {};

\node (7) at (370.000000:1.550000) {\scalebox{2}{$\color{black} \cdot$}};
\node (11) at (360.000000:1.550000) {\scalebox{2}{$\color{black} \cdot$}};
\node (12) at (380.000000:1.550000) {\scalebox{2}{$\color{black} \cdot$}};

\node [myn,circle,label=410.000000:$v_{2^{n+1}}$,fill=black!30,minimum size=0.52cm] 
(8) at (410.000000:1.50000) {};

\draw[very thick,black][<-] (0) -- (5);
\draw[very thick,black][->] (8) -- (3);
\DoubleLine{0}{1}{<-,very thick}{}{->,very thick}{}
\DoubleLine{0}{8}{<-,very thick}{}{->,line width = 2.0pt}{}
\DoubleLine{3}{4}{<-,very thick}{}{->,very thick}{}
\DoubleLine{4}{5}{<-,very thick}{}{->,very thick}{}
\DoubleLine{5}{6}{<-,very thick}{}{->,very thick}{}

\end{tikzpicture}
\end{center}
\endpgfgraphicnamed
\caption{The directed Haj\'os join of two copies of the digraph $H_0$ and the resulting digraph $H_1 = (H_0,v_0,v_{2^n})  \otimes (H'_0,v'_{2^n+1},v'_0)$.}
\label{Fig_H1}
\end{figure}
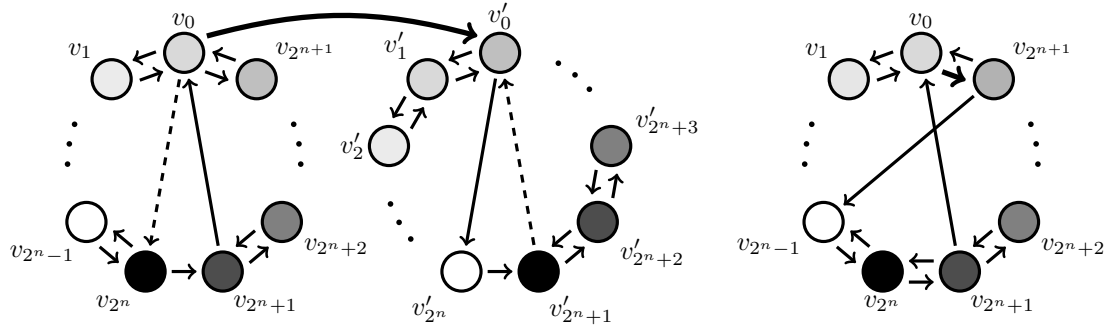

In order to simplify the writing we relabel (in cyclic order) the vertices of $H_1$, for $i = 0,1,...,2^{n+1}$, $v_i: = v_{i+1}$ the index are taken {modulo} $2^{n+1}+1$. 

For $k=1,\dots,n-1$, let $H_{k+1}=H_k\otimes H'_k$.
By Lemma \ref{Lem}, $H_{k+1}$ is a symmetric cycle of order $2^{n+1}+1$ with asymmetric arcs $v_0v_{2^n}$ and $v_{2^n+2^k+1}v_{2^k}$. 
Applying Lemma \ref{Lem}, the digraph $H_{n+1}$ is a symmetric cycle of order $2^{n+1}+1$ with a unique asymmetric arc $v_0v_{2^n}$ because  ${2^n+2^n+1}\equiv 0\mod 2^{n+1}+1$. 

Finally, let $H_{n+2} = (H_{n+1},v_0,v_{2^{n}})  \otimes (H'_{n+1},v'_0,v'_{2^n})$. 
By Remark 1,
 the vertex $v'_{2^n}\in V(\hat{H}_{n+2})$ is identified with the vertex $v_{2^n-(0)+2^n} = v_{2^{n+1}}$, thus the arc $v_{0}v'_{2^{n}}\in A(\hat{H}_{n+2})$ is transformed into the arc $v_{0}v_{2^{n+1}}\in A(H_{n+2})$. Therefore, $H_{n+2} = D(C_{2^{n+1}+1})$.

Observe that we use only one directed Haj\'os  operation (a directed Haj\'os  join) to obtain the digraph $H_0$.
For each of the digraphs $H_1,H_2,\dots, H_{n+2}$ we use $2^{n+1}+1$ directed Haj\'os  operations (a directed Haj\'os  join and $2^{n+1}$ identifications), thus we use $(n+2)(2^{n+1}+1)+1$ directed Haj\'os  operations. 
\end{proof}

The Haj\'os number of an $r$-dichromatic digraph $H$ was defined in \cite{MR0671913} as the minimum number of Haj\'os operations needed to obtain $H$ from $D(K_r)$. 

\begin{teo}\label{TeoNumHaj}
Let $n\ge2$. The Haj\'os  number of a symmetric cycle of order $2^{n}+1$ is at most  $$\dfrac{n\left(2^{n+2}+n+5\right)}{2}-7.$$
Moreover, the Haj\'os  number of a symmetric cycle of order $2m+1$, where $2^{n-1}+1<2m+1<2^n+1$, is at most 
$$\dfrac{n\left(2^{n+2}+n+5\right)}{2}-5.$$
\end{teo}

\begin{proof}

By recursively applying Theorem \ref{TheoNextCycle}, the number of steps to build the symmetric cycle of order $2^n+1$ from $D(K_3)$ is at most:
$$\sum\limits_{i=2}^n(i+1)\left(2^{i}+1\right)+1. $$

Note that 
$$\begin{array}{lcl}
\sum\limits_{i=1}^ni2^{i}&=&n(2^{n+1}-2)- \sum\limits_{j=1}^{n-1}\left(2^{j+1}-2\right)\\
& = & (n-1)\left(2^{n+1}\right)+2.\\
\end{array}
$$
Thus, 
$$\begin{array}{lcl}
\sum\limits_{i=1}^n(i+1)\left(2^{i}+1\right)+1  
& = & \sum\limits_{i=1}^ni2^{i}+ \sum\limits_{i=1}^ni + \sum\limits_{i=1}^n2^{i} + \sum\limits_{i=1}^n1 + \sum\limits_{i=1}^n1\\\\
& = & (n-1)2^{n+1}+2 + \frac{n(n+1)}{2}+(2^{n+1}-2)+2n\\\\
%& = & \frac{n(n+5)}{2}+n2^{n+1}\\\\
& = & \frac{n}{2}\left(2^{n+2}+n+5\right).\end{array}
$$
Since 
$$\sum\limits_{i=2}^n(i+1)\left(2^{i}+1\right)+1  =  \sum\limits_{i=1}^n(i+1)\left(2^{i}+1\right)+1 -7, $$
the first result follows.

Consider a symmetric cycle of order $2^n+1$. Let $1\le m< 2^{n-1}$, let $V_{2m}=\{v_{2m}, v_{2m+2},\dots,v_{2^n}\}$ and let $V_{2m+1}=\{v_{2m+1}, v_{2m+3},\dots,v_{2^n+1}\}$, note that $V_{2m}$ and $V_{2m+1}$ are both independent sets of vertices.
Identifying the independent set of vertices $V_{2m}$ into the vertex $v_{2m}$ and identifying the independent set of vertices $V_{2m+1}$  into the vertex $v_{2m+1}$, we obtain the symmetric cycle of order $2m+1$. Thus, we can construct the symmetric cycle of order $2m+1$ using $\dfrac{n\left(2^{n+2}+n+5\right)}{2}-5$ directed Haj\'os  operations and the result follows.
\end{proof}

\section{Computational Complexity}
In this section we determine the complexity of our procedure and which is an upper bound for the complexity of constructing a symmetric odd cycle using directed Haj\'os operations. 
\begin{teo}
Let $n\ge5$ be an odd integer. % and let $D(C_{n})$ be an odd cycle. 
The complexity of the procedure to obtain  $D(C_{n})$ from $D(K_3)$ is $\Theta(n\ln(n))$.
\end{teo}
\begin{proof}
Let $m\geq2$ and let $n$ be an odd integer such that $n_1< n\leq n_2$ with
$$
n_1 = 2^{m-1}+1~ and~ n_2 = 2^m+1.
$$

With our procedure, one needs
%Con nuestro procedimiento, se necesitan 
\begin{equation}\label{steps}
X = \displaystyle\frac{m}{2}(2^{m+2}+m+5)- \alpha
\end{equation} 
directed Haj\'os  operations to obtain a symmetric cycle of order $n$ from $D(K_3)$, where

$$\alpha = \left \{\begin{array}{ccc}
5 & \text{ if } & n < n_2, \\
7 & \text{ if } & n = n_2. \\
\end{array}\right.$$ 
Since $n_1 = 2^{m-1} + 1$, then
$$\begin{array}{rcl}
n_1 -1 & = & 2^{m-1} \\
\log_{2}(n_1 - 1) & = &  m-1, \\
\end{array}$$ 
therefore
\begin{equation}\label{m}
m = \log_{2}(n_1 - 1)+1.
\end{equation}
In equation \ref{steps} we have that
$$\begin{array}{rcl}
X  & = & \displaystyle\frac{\log_{2}(n_1 - 1)+1}{2}(2^{(\log_{2}(n_1 - 1)+1)+2}+(\log_{2}(n_1 - 1)+1)+5)- \alpha\\\\ 
  & = & \displaystyle\frac{\log_{2}(n_1 - 1)+1}{2}(8n_1 +\log_{2}(n_1 - 1) -2)- \alpha \\\\
 & = &  4n_1\log_{2}(n_1 - 1)
+ \displaystyle\frac{(\log_{2}(n_1 - 1))^2}{2} 
- \log_{2}(n_1 - 1)\\\\
 &  & + 4n_1 + \displaystyle\frac{\log_{2}(n_1 - 1)}{2} - 1 - \alpha \\\\
 & < & 9n_1\log_{2}(n_1 - 1)~ <~ 9n\log_{2}(n) ~ = ~ \displaystyle\left(\frac{9}{\ln(2)}\right)n\ln(n)~ <~ 13n\ln(n).   \\\\ 
\end{array}$$ 
On the other hand, we have
$$\begin{array}{rcl}
 X & = &  4n_1\log_{2}(n_1 - 1)
+ \displaystyle\frac{(\log_{2}(n_1 - 1))^2}{2} 
- \log_{2}(n_1 - 1)\\\\
 &  & + 4n_1 + \displaystyle\frac{\log_{2}(n_1 - 1)}{2} - 1 - \alpha \\\\
 & > & 4n_1\log_{2}(n_1 - 1) \\\\ 
  & > & 2n_1\log_{2}((n_1-1)^2)~ =~ (2^m+2)\log_{2}(2^{m+1}) \\\\ 
  & > & (2^m+1)\log_{2}(2^{m} + 1)~ =~ n_2\log_{2}(n_2)\\\\  
  & \geq & n\log_{2}(n)~ >~  n\ln(n).\\\\  
\end{array}$$ 

Therefore $n\ln(n) < X < 13n\ln(n)$ and the complexity of our procedure is $\Theta(n\ln(n))$.
\end{proof}

%\section{Acknowledgements}
%Juan Carlos was supported by Conacyt Grant 963921. This article was supported by Conacyt Project 47510664 and Department of Applied Mathematics and Systems (DMAS) at Metropolitan Autonomous University (UAM) Cuajimalpa.

\section{Conclusions}

Although it was proved in \cite{BangJensen2020} that any 3-critical digraph can be constructed by a sequence of Haj\'os operations, it is not a trivial task to obtain such a sequence even for simple digraphs such as symmetric cycles of odd length and in particular, the symmetric cycle of length 5. 
Using genetic algorithm, we obtained in  \cite{GarciaAlta},  a sequence of Haj\'os operations for the symmetric cycle of length 5, and generalizing this particular result permitted  us to construct any symmetric odd cycle using Haj\'os operations.  

We believe that these ideas can be used in order to obtain the 3- and 4-critical tournaments characterized in \cite{Neumann1994} and hopefully a improvement of the upper bound for the minimum order of a 5-critical tournament, which is known to be 19 (see 4.6. An application \cite{NeumannLaraZykov}).

\end{document}